\documentclass[a4paper]{article}
\usepackage{amsmath,latexsym}
\usepackage{amsfonts,amssymb}
\usepackage{amsthm}
\usepackage{inputenc}
\usepackage[mathcal]{euscript}
\usepackage{graphicx}
\usepackage{epstopdf}
\usepackage{float}

\newtheorem{theorem}{Theorem}[section]
\newtheorem{corollary}[theorem]{Corollary}
\newtheorem{lemma}[theorem]{Lemma}

\newtheorem{definition}[theorem]{Definition}
\newtheorem{remark}[theorem]{Remark}
\newtheorem{proposition}[theorem]{Proposition}

\numberwithin{equation}{section}

\def\bbf{{\mathbb F}}

\def\bn{{\mathbb N}}

\def\bq{{\mathbb Q}}
\def\br{{\mathbb R}}

\def\bz{{\mathbb Z}}

\title{Local Descriptions of Roots of Cubic Equations over P-adic Fields}

\author{Mansoor Saburov and Mohd Ali Khameini Ahmad}

\begin{document}
\maketitle

%\address{Mansoor Saburov\\
%Department of Computational \& Theoretical Sciences \\
%Faculty of Science, International Islamic University Malaysia\\
%P.O. Box, 25200, Kuantan\\
%Pahang, Malaysia} \email{{\tt msaburov@@gmail.com}}

%\author{}
%\address{MOHD ALI KHAMEINI AHMAD\\
%Department of Computational \& Theoretical Sciences \\
%Faculty of Science, International Islamic University Malaysia\\
%P.O. Box, 25200, Kuantan\\
%Pahang, Malaysia} \email{{\tt khameini.ahmad@@gmail.com}}

\begin{abstract}
The most frequently asked question in the $p-$adic lattice models of statistical mechanics is that whether a root of a polynomial equation belongs to domains $\mathbb{Z}_p^{*}, \ \mathbb{Z}_p\setminus\mathbb{Z}_p^{*}, \ \mathbb{Z}_p, \ \mathbb{Q}_p\setminus\mathbb{Z}_p^{*}, \ \mathbb{Q}_p\setminus\left(\mathbb{Z}_p\setminus\mathbb{Z}_p^{*}\right), \ \mathbb{Q}_p\setminus\mathbb{Z}_p, \ \mathbb{Q}_p $ or not. However, this question was open even for lower degree polynomial equations. In this paper, we give local descriptions of roots of cubic equations over the $p-$adic fields for $p>3$.


\end{abstract}

\section{Introduction}	

The fields of $p-$adic numbers were introduced by German
mathematician K. Hensel \cite{Hen}. The $p-$adic numbers were
motivated primarily by an attempt to bring the ideas and techniques
of the power series into number theory. Their canonical
representation is analogous to the expansion of analytic functions
into power series. This is one of the manifestations of the analogy
between algebraic numbers and algebraic functions. Over the last century, $p-$adic numbers and $p-$adic analysis have
come to play a central role in modern number theory. This importance
comes from the fact that they afford a natural and powerful language
for talking about congruences between integers, and allow using the
methods borrowed from analysis for studying such problems.
 More recently, numerous applications of $p-$adic numbers have shown
up in theoretical physics and quantum mechanics (see for example,
\cite{FreWit}, \cite{Khren91,Khren94}, \cite{MarPar}-\cite{MH2014}, \cite{VladVolZel,Vol}).

Finding roots of polynomials is among the old problem of mathematics. In the filed of real numbers, this problem found its own solution. However, to the best of our knowledge, in the field of $p-$adic numbers -- in the counterpart of the field of real number, the less attention was paid for this problem in the literature. Recently, by concerning with some problems of $p-$adic lattice models of statistical mechanics, this problem is again raised up. The scenario is completely different from the field of real numbers to the field of $p-$adic numbers. For instance, the quadratic equation $x^2+1=0$ is not solvable in the real field but solvable in the  $p-$adic field for $p\equiv 1 \ (mod \ 4)$. Vise versa, the cubic equation $x^3+p=0$ is not solvable in the $p-$adic field but solvable in the real field. Therefore, it is of independent interest to provide a solvability criterion for lower degree polynomial equations over the $p-$adic field. The solvability criterion for quadratic equations over the $p-$adic field was provided in all classical $p-$adic analysis books. Recently, in the series of papers \cite{FMBOMS,FMBOMSKM}, \cite{SMAA2013}-\cite{SMAA2014b}, the solvability criterion and the number of roots of cubic equations over the $p-$adic field were studied. This paper is a continuation of previous studies and we are aiming to locally describe all roots of cubic equations over the $p-$adic field for $p>3$. Applications of quadratic and cubic equations in the $p-$adic lattice models of statistical mechanics were presented in the papers \cite{M2}-\cite{MH2014}, \cite{RK}, \cite{SMAA2014e}. The solvability criterion and the number of roots of bi-quadratic equations are also studied in the papers \cite{SMAA2014c,SMAA2014d}.

\section{Preliminaries}

For a fixed prime $p$, $\bq_p$ denotes the field of $p-$adic
numbers, it is a completion of the rational numbers $\bq$ with
respect to the non-Archimedean norm $|\cdot|_p:\bq\to\br$ given by
\begin{eqnarray}
|x|_p=\left\{
\begin{array}{c}
  p^{-r} \ x\neq 0,\\
  0,\ \quad x=0,
\end{array}
\right.
\end{eqnarray}
here, $x=p^r\frac{m}{n}$ with $r,m\in\bz,$ $n\in\bn$,
$(m,p)=(n,p)=1$. A number $r$ is called \textit{a $p-$order} of $x$
and it is denoted by $ord_p(x)=r.$

Any $p-$adic number $x\in\bq_p$ can be uniquely represented as the
following canonical form
\begin{eqnarray*}
x=p^{ord_p(x)}\left(x_0+x_1\cdot p +x_2\cdot p^2+\cdots \right)
\end{eqnarray*}
where $x_0\in \{1,2,\cdots p-1\}$ and $x_i\in\{0,1,2,\cdots p-1\}$,
$i\geq 1,$ (see \cite{Bor Shaf}, \cite{NK})

We respectively denote the set of all {\it $p-$adic integers} and {\it units} of
$\bq_p$ by
$$\bz_p=\{x\in\bq_{p}: |x|_p\leq1\}, \quad \bz_p^{*}=\{x\in\bq_{p}: |x|_p=1\}.$$

Any nonzero $p-$adic number $x\in\bq_p$ has a unique representation of the form $x =\cfrac{x^{*}}{|x|_p}$, where $x^{*}\in\bz_p^{*}$.

Let $p$ be a prime number, $\mathbb{F}_p=\{\bar{0},\bar{1},\cdots, \overline{p-1}\}$ be a finite field, $q\in\bn$, $a\in\bbf_p$ with $a\neq \bar{0}.$  The number $a$ is called \textit{the $q$-th power
residue modulo $p$} if the the following equation
\begin{eqnarray}\label{kthresidue}
x^q=a
\end{eqnarray} has a solution in $\bbf_p$.

\begin{proposition}[\cite{Ros}]\label{aisresidueofp}
Let $p$ be an odd prime, $q\in\bn$, $d=(q,p-1),$ and
$a\in\bbf_p$ with $a\neq\bar{0}.$ Then the following statements
hold:
\begin{itemize}
  \item [(i)] $a$ is the $q$-th power residue modulo $p$ if and only if one has
$a^{\frac{p-1}{d}}=\bar{1};$
  \item [(ii)] If $a^{\frac{p-1}{d}}=\bar{1}$ then the equation \eqref{kthresidue} has $d-$number of solutions in $\bbf_p$.
\end{itemize}
\end{proposition}

The solvability criterion for the following monomial equation in $\bq_p$
\begin{eqnarray}\label{x^q=a}
x^q=a,
\end{eqnarray}
where $q\in\bn$, $a\in \bq_p$ with $a\neq 0$, was provided in \cite{FMMS}.

\begin{proposition}[\cite{FMMS}]\label{Criterionforp}
Let $p$ be an odd prime, $q\in\bn,$ $a\in \bq_p,$
$a=\frac{a^{*}}{|a|_p}$ and $a^{*}\in\bz_p^{*}$ with
$a^{*}=a_0+a_1\cdot p+\cdots$. Then the following
statements hold true:
\begin{itemize}
  \item [(i)] If $(q,p)=1$ then the equation \eqref{x^q=a} is solvable in $\bq_p$ if and only if $a_0^{\frac{p-1}{(q,p-1)}}\equiv 1 \ ( mod \ p)$ and $\log_p{|a|_p}$ is divisible by $q$
  \item [(ii)] If $(q,p)=1$, $a_0^{\frac{p-1}{(q,p-1)}}\equiv 1 \ ( mod \ p)$ and $\log_p{|a|_p}$ is divisible by $q$ then the equation \eqref{x^q=a} has $(q,p-1)-$number of roots in $\bq_p$.
  \item [(iii)] If $(q,p)=1$ then the equation $y^q\equiv a^{*} \ (mod \ p^{k})$ has a root in $\bz_p^{*}$ for some $k\in\bn$ if and only if $a_0^{\frac{p-1}{(q,p-1)}}\equiv 1 \ ( mod \ p).$
  \item [(iv)] If $q=m\cdot p^s$ with $(m,p)=1$, $s\geq 0$ then the equation \eqref{x^q=a}
  is solvable in $\bq_p$ if and only if $a_0^{\frac{p-1}{(m,p-1)}}\equiv 1 \ ( mod \ p)$, $a^{p^s}_0\equiv a \ (mod \ p^{s+1})$ and $\log_p{|a|_p}$ is divisible by $q$.
\end{itemize}
\end{proposition}

Throughout this paper, we always assume that $p>3$ unless otherwise mentioned. 

Let $a\in \bq_p$ be a nonzero $p-$adic number such that $a=\frac{a^{*}}{|a|_p}$ and $a^{*}\in\bz_p^{*}$ with
$a^{*}=a_0+a_1\cdot p+a_2\cdot p^2+\cdots$.

\begin{definition}
We say that there exists $\sqrt[q]{a}$, written $\sqrt[q]{a}-\exists,$ if the monomial equation \eqref{x^q=a} is solvable in $\mathbb{Q}_p$. Particularly, we say that there exists $\sqrt{a}$, written $\sqrt{a}-\exists,$ if $a_0^{\frac{p-1}{2}}\equiv 1 \ ( mod \ p)$ and $\log_p{|a|_p}$ is even. Similarly, we say that there exists $\sqrt[3]{a}$, written $\sqrt[3]{a}-\exists,$ if $a_0^{\frac{p-1}{(3,p-1)}}\equiv 1 \ ( mod \ p)$ and $\log_p{|a|_p}$ is divisible by 3. 
\end{definition}

\section{The Main Problems and Results}\label{ProblemsandResults}

Generally, we may come across the following problem in one form or another: \textit{provide a solvability criterion for the polynomial equation over the given set $\mathbb{A}\subset\mathbb{Q}_p$.} 

In this paper, we study the depressed cubic equation
\begin{equation}\label{1.1}
x^3+ax=b
\end{equation}
over $\mathbb{Q}_p$, where $a,b \in \mathbb{Q}_p$, $ab\neq 0$, and $p>3$. 

Let $\mathbb{A}, \ \mathbb{B}\subset\mathbb{Q}_p$ be two nonempty disjoint sets, i.e., $\mathbb{A}\cap \mathbb{B}=\emptyset$

\begin{definition}[Solvability in $\mathbb{A}$]
We say that the cubic equation \eqref{1.1} is solvable in $\mathbb{A}$ if at least one of its roots belongs to $\mathbb{A}$.
\end{definition}

\begin{definition}[Solvability in $\mathbb{A}^{[n]}$]
We say that the cubic equation \eqref{1.1} is solvable in $\mathbb{A}^{[n]}$ if it has exactly  $n-$number of roots (including multiplicity) and all of them belong to the set $\mathbb{A}$, where $n=1$ or $3$.
\end{definition}

\begin{definition}[Solvability in $\mathbb{A}^{[n]} \sqcup \mathbb{B}^{[m]}$]
We say that the cubic equation \eqref{1.1} is solvable in $\mathbb{A}^{[n]} \sqcup \mathbb{B}^{[m]}$ if the $n-$number of its roots (including multiplicity) belong to $\mathbb{A}$ and the $m-$number of its roots (including multiplicity) belong to $\mathbb{B}$, where $m+n=3$ and $m,n>0$.
\end{definition}

The main problems of the paper are the following

\begin{itemize}
\item[1.] Provide solvability criteria for the cubic equation \eqref{1.1} in $\mathbb{A}$, where $$\mathbb{A}\in\left\{\mathbb{Z}_p^{*}, \ \ \mathbb{Z}_p\setminus\mathbb{Z}_p^{*}, \ \ \mathbb{Z}_p, \ \ \mathbb{Q}_p\setminus\mathbb{Z}_p^{*}, \ \ \mathbb{Q}_p\setminus\left(\mathbb{Z}_p\setminus\mathbb{Z}_p^{*}\right), \ \ \mathbb{Q}_p\setminus\mathbb{Z}_p, \ \ \mathbb{Q}_p\right\};$$
\item[2.] Provide the number $\mathbf{N}_{\mathbb{A}}(x^2+ax-b)$ of roots of the cubic equation \eqref{1.1} where
$$\mathbb{A}\in\left\{\mathbb{Z}_p^{*}, \ \ \mathbb{Z}_p\setminus\mathbb{Z}_p^{*}, \ \ \mathbb{Z}_p, \ \ \mathbb{Q}_p\setminus\mathbb{Z}_p^{*}, \ \ \mathbb{Q}_p\setminus\left(\mathbb{Z}_p\setminus\mathbb{Z}_p^{*}\right), \ \ \mathbb{Q}_p\setminus\mathbb{Z}_p, \ \ \mathbb{Q}_p\right\};$$
\item[3.] Provide solvability criteria for the cubic equation \eqref{1.1} in $\mathbb{A}^{[n]}$, \ $\mathbb{A}^{[n]} \sqcup \mathbb{B}^{[m]}$ where
 $$\mathbb{A}, \ \mathbb{B}\in\left\{\mathbb{Z}_p^{*}, \ \ \mathbb{Z}_p\setminus\mathbb{Z}_p^{*}, \ \ \mathbb{Q}_p\setminus\mathbb{Z}_p\right\};$$
\end{itemize}

The solution of the third problem gives a local description of roots of the cubic equation \eqref{1.1}.  In order to solve it, we have to solve the first and second problems for domains $\mathbb{Z}_p^{*}, \ \mathbb{Z}_p\setminus\mathbb{Z}_p^{*}, \ \mathbb{Q}_p\setminus\mathbb{Z}_p.$ The first and second problems for domains $\mathbb{Z}_p^{*}, \ \mathbb{Z}_p, \ \mathbb{Q}_p$ were already studied in the paper \cite{FMBOMS}. Therefore, it is enough to study the first and second problems in domains $\mathbb{Z}_p\setminus\mathbb{Z}_p^{*}, \ \mathbb{Q}_p\setminus\mathbb{Z}_p.$ For the quadratic equation, all these problems were studied in the paper \cite{SMAA2014e}. 

The solvability criterion for the cubic equation \eqref{1.1} in $\mathbb{Q}_p$ was given  in \cite{FMBOMS}.

Since $ab\neq 0,$ we have that $a=\cfrac{a^{*}}{|a|_p}$ and $b=\cfrac{b^{*}}{|b|_p}$, where
\begin{eqnarray*}
a^{*}&=&a_0+a_1p+a_2p^2+\cdots\\
b^{*}&=&b_0+b_1p+b_2p^2+\cdots
\end{eqnarray*}
where $a_0,b_0\in\{1,2,\cdots p-1\}$ and $a_i,b_i\in\{0,1,2,\cdots p-1\}$ for any $i\in\mathbb{N}.$
We set $D_0=-4a_0^3-27b_0^2$ and $u_{n+3}=b_0u_n-a_0u_{n+1}$ with $u_1=0,$ $u_2=-a_0,$ and $u_3=b_0$ for $n=\overline{1,p-3}$. 

%Recall that $\sqrt[3]{b}$ exists if and only if  $b_0^{\frac{p-1}{(3,p-1)}}\equiv 1 \ (mod \ p)$ and $\log_p|b|_p$ is divisible by 3. We shall use the notation $\sqrt[3]{b}-\exists$ whenever there exists $\sqrt[3]{b}$. Meanwhile, there exists $\sqrt{-a}$ if and only if  $(-a_0)^{\frac{p-1}{2}}\equiv 1 \ (mod \ p)$ and $\log_p|a|_p$ is even. We shall use the notation $\sqrt{-a}-\exists$ whenever there exists $\sqrt{-a}$.

\begin{theorem}[Solvability Domain, \cite{FMBOMS}]\label{MainResult1}
Let $p>3$. The cubic equation \eqref{1.1} is solvable in $\mathbb{Q}_p$ if and only if either one of the following conditions holds true:
\begin{itemize}
\item [$1.$] $|a|_p^3<|b|_p^2,$  and \  $\sqrt[3]{b}-\exists;$
\item [$2.$] $|a|_p^3=|b|_p^2$ and $D_0u_{p-2}^2\not\equiv 9a_0^{2} \ (mod \ p);$
\item [$3.$] $|a|_p^3>|b|_p^2.$
\end{itemize}
\end{theorem}

Let us define the following sets
\begin{eqnarray*}
\Delta_1&=&\left\{(a,b)\in \mathbb{Q}_p\times\mathbb{Q}_p: \ |a|_p^3<|b|_p^2, \  \sqrt[3]{b}-\exists\right\}\\
\Delta_2&=&\left\{(a,b)\in \mathbb{Q}_p\times\mathbb{Q}_p: \ |a|_p^3=|b|_p^2, \ D_0u_{p-2}^2\not\equiv 9a_0^{2} \ (mod \ p) \right\}\\
\Delta_3&=&\left\{(a,b)\in \mathbb{Q}_p\times\mathbb{Q}_p: \ |a|_p^3>|b|_p^2\right\}\\
\Delta&=&\Delta_1\cup\Delta_2\cup\Delta_3
\end{eqnarray*}

The set $\Delta\subset \mathbb{Q}_p\times\mathbb{Q}_p$ is called \textit{a solvability domain} of the depressed cubic equation \eqref{1.1}. Since $\mathbb{Q}_p$ is a disordered field, we could not describe the solvability domain $\Delta$ in the picture. However, we can describe the $p-$adic absolute value of elements of the set $\Delta$ in the picture. It was presented in Fig. \ref{figSolvability domain}. We refer it as the solvability domain of the cubic equation \eqref{1.1}.

\begin{figure}[H]
	\centering
		\includegraphics[width=0.6\textwidth]{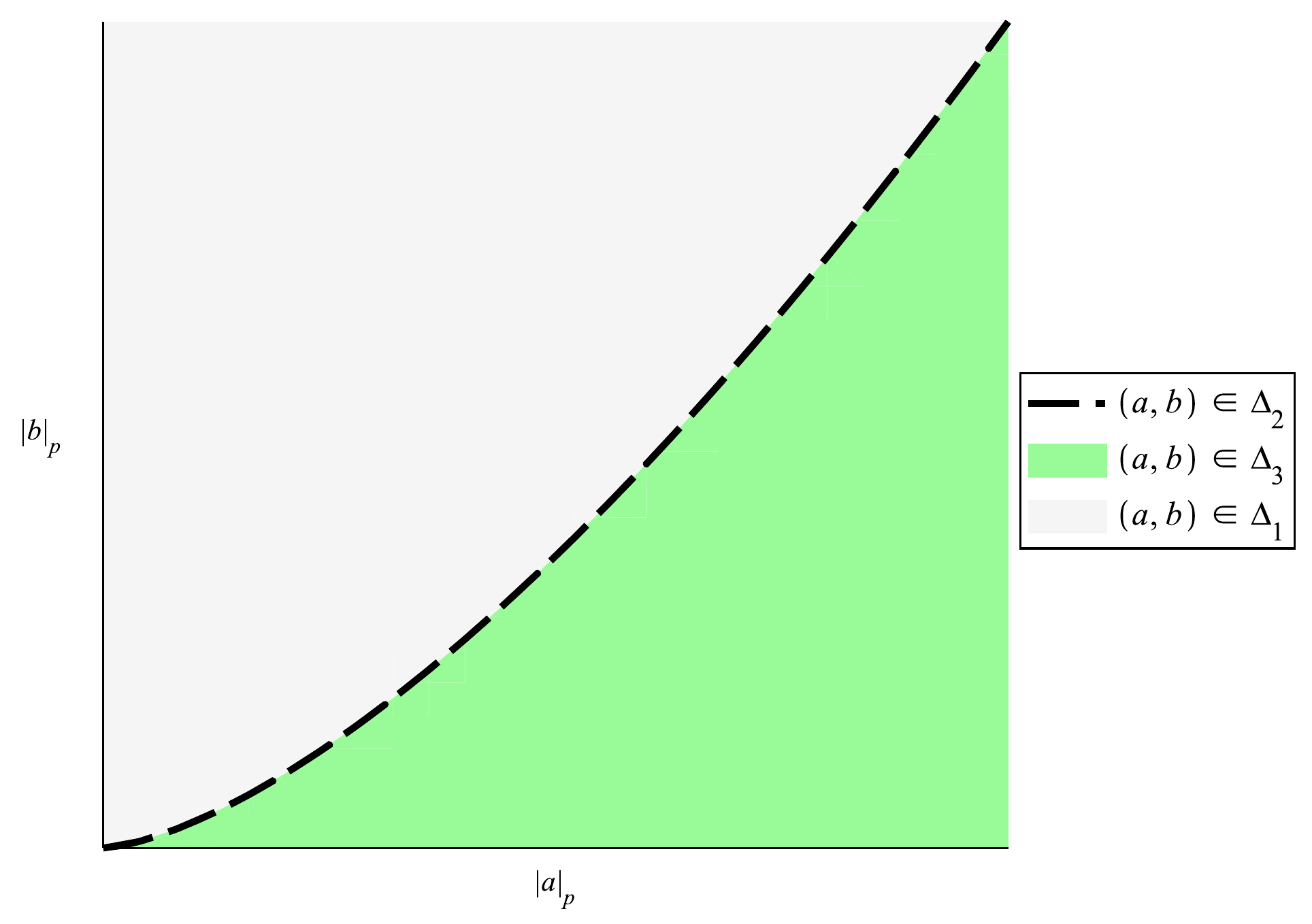}
		\rule{35em}{0.5pt}
	\caption[The solvability domain]{The solvability domain of the cubic equation \eqref{1.1}}
	\label{figSolvability domain}
\end{figure}

The following results are the main results of the paper which give a full description of $p-$adic absolute values of roots of the cubic equation \eqref{1.1} in the solvability domain $\Delta$.

\begin{theorem}[Descriptions of Roots in $\Delta_1\cup\Delta_3$]\label{MainResult2}
Let $p>3$ and $(a,b)\in\Delta_1\cup\Delta_3$. The cubic equation \eqref{1.1} is 
\begin{itemize}
\item[$1$.] Solvable in ${\mathbb{Z}^*_p}^{[3]}$ if and only if $$(a,b)\in\Delta_1, \ |b|_p = 1,\ p \equiv 1 \ (mod \ 3);$$
\item[$2$.] Solvable in ${\left(\mathbb{Z}_p \setminus \mathbb{Z}^{*}_p\right)}^{[3]}$ if and only if 
$$(a,b)\in\Delta_1, \ |b|_p<1, \ p \equiv 1 \ (mod \ 3)\ \ \text{or} \ \ (a,b)\in\Delta_3, \ |a|_p<1, \ \sqrt{-a}-\exists;$$
\item[$3$.] Solvable in ${\left(\mathbb{Q}_p\setminus\mathbb{Z}_p\right)}^{[3]}$ if and only if 
$$(a,b)\in\Delta_1, \ |b|_p>1, \ p \equiv 1 \ (mod \ 3) \ \ \text{or} \ \ (a,b)\in\Delta_3, \ |b|_p>|a|_p, \ \sqrt{-a}-\exists;$$
\item[$4$.] Solvable in $(\mathbb{Z}_p\setminus\mathbb{Z}^*_p)^{[1]} \sqcup (\mathbb{Z}_p^*)^{[2]}$ if and only if 
$$(a,b)\in\Delta_3,  \ |a|_p =1, \ \sqrt{-a}-\exists;$$
\item[$5$.] Solvable in ${\mathbb{Z}_p^*}^{[1]} \sqcup (\mathbb{Q}_p\setminus\mathbb{Z}_p)^{[2]}$ if and only if 
$$(a,b)\in\Delta_3, \ |a|_p = |b|_p, \ \sqrt{-a}-\exists;$$
\item[$6$.] Solvable in $(\mathbb{Z}_p\setminus\mathbb{Z}^*_p)^{[1]} \sqcup (\mathbb{Q}_p\setminus\mathbb{Z}_p)^{[2]}$ if and only if 
$$(a,b)\in\Delta_3, \ |a|_p > |b|_p, \ |a|_p > 1, \ \sqrt{-a}-\exists;$$
\item[$7$.] Solvable in ${\mathbb{Z}^*_p}^{[1]}$ if and only if 
$$(a,b)\in\Delta_1, \ |b|_p = 1, \ p \equiv 2 \ (mod \ 3), \  \text{or} \  (a,b)\in\Delta_3, \ |a|_p=|b|_p, \ \sqrt{-a}-\not\exists;$$
\item[$8$.] Solvable in $(\mathbb{Z}_p \setminus \mathbb{Z}^{*}_p)^{[1]}$ if and only if  
$$(a,b)\in\Delta_1, |b|_p<1, \ p \equiv 2 \ (mod \ 3) \ \ \text{or} \ \  (a,b)\in\Delta_3, \ |a|_p>|b|_p, \ \sqrt{-a}-\not\exists;$$
\item[$9$.] Solvable in $(\mathbb{Q}_p\setminus\mathbb{Z}_p)^{[1]}$ if and only if  $$(a,b)\in\Delta_1, \ |b|_p>1, \ p \equiv 2 \ (mod \ 3) \ \ \text{or} \ \ (a,b)\in\Delta_3, \ |a|_p<|b|_p, \ \sqrt{-a}-\not\exists.$$
\end{itemize}
\end{theorem}

\begin{theorem}[Descriptions of Roots in $\Delta_2$]\label{MainResult3}
Let $p>3$ and $(a,b)\in\Delta_2$. The cubic equation \eqref{1.1} is 
\begin{itemize}
\item[$1$.] Solvable in ${\mathbb{Z}^*_p}^{[3]}$ if and only if $|b|_p = 1$ and
$$ |D|_p=1, \ D_0u_{p-2}^2 \equiv 0 \ (mod \ p) \quad \text{or} \quad 0 \leq |D|_p < 1, \ \sqrt{D}-\exists;$$
\item[$2$.] Solvable in $(\mathbb{Z}_p \setminus \mathbb{Z}^{*}_p)^{[3]}$ if and only if $|b|_p<1$ and 
$$|D|_p=1, \ D_0u_{p-2}^2 \equiv 0 \ (mod \ p) \quad \text{or} \quad 0 \leq |D|_p < 1, \ \sqrt{D}-\exists;$$
\item[$3$.] Solvable in $(\mathbb{Q}_p\setminus\mathbb{Z}_p)^{[3]}$ if and only if $|b|_p>1$ and 
$$|D|_p=1,\ D_0u_{p-2}^2 \equiv 0 \ (mod \ p) \quad \text{or} \quad 0 \leq |D|_p < 1, \ \sqrt{D}-\exists;$$
\item[$4$.] Solvable in ${\mathbb{Z}^*_p}^{[1]}$ if and only if $|b|_p = 1$ and 
$$|D|_p=1,\ D_0u_{p-2}^2 \not\equiv 0,9a_0^2 \ (mod \ p) \quad \text{or} \quad 0 < |D|_p < 1, \ \sqrt{D}-\not\exists;$$
\item[$5$.] Solvable in $(\mathbb{Z}_p \setminus \mathbb{Z}^{*}_p)^{[1]}$ if and only if $|b|_p<1$ 
$$|D|_p=1, \ D_0u_{p-2}^2 \not\equiv 0,9a_0^2 \ (mod \ p) \quad \text{or} \quad 0 < |D|_p < 1, \ \sqrt{D}-\not\exists;$$
\item[$6$.] Solvable in $(\mathbb{Q}_p\setminus\mathbb{Z}_p)^{[1]}$ if and only if $|b|_p>1$ and 
$$|D|_p=1, \ D_0u_{p-2}^2 \not\equiv 0,9a_0^2 \ (mod \ p) \quad \text{or} \quad 0 < |D|_p < 1, \ \sqrt{D}-\not\exists.$$
\end{itemize}
\end{theorem}

\begin{remark}
If $(a,b)\in \Delta$ then the cubic equation \eqref{1.1} is not solvable in $(\mathbb{Z}_p\setminus\mathbb{Z}^{*}_p) \ \sqcup \ \mathbb{Z}_p^{*} \  \sqcup \ (\mathbb{Q}_p\setminus\mathbb{Z}_p).$ In other words, for any  $(a,b)\in \Delta$, there is no cubic equation \eqref{1.1} in which its three solutions (if any) belong to different domains $\mathbb{Z}_p\setminus\mathbb{Z}^{*}_p, \ \mathbb{Z}_p^{*}, \ \mathbb{Q}_p\setminus\mathbb{Z}_p.$ Moreover, if $(a,b) \in \Delta_2$ then the cubic equation \eqref{1.1} is not solvable in $\mathbb{A}^{[n]} \sqcup \mathbb{B}^{[m]}$ where $\mathbb{A},\mathbb{B} \in \{\mathbb{Z}_p\setminus\mathbb{Z}_p^{*},\ \mathbb{Z}_p^{*},\ \mathbb{Q}_p\setminus\mathbb{Z}_p\}, \ m+n=3, \ m,n>0$.
\end{remark}

The graphical illustration of Theorem \ref{MainResult2} is given in Figs. \ref{figDescription1}-\ref{figDescription4}. The graphical illustration of Theorem \ref{MainResult3} is given in Figs. \ref{figDescription5}-\ref{figDescription6}. The main strategy of proving Theorem \ref{MainResult2}-\ref{MainResult3} is presented in the next section. Namely, Theorem \ref{MainResult2}-\ref{MainResult3} follow from Theorem \ref{Numbers} which describes the number $\mathbf{N}_{\mathbb{Z}^*_p}(x^3+ax-b)$, $\mathbf{N}_{\mathbb{Z}_p\setminus\mathbb{Z}^*_p}(x^3+ax-b)$ and $\mathbf{N}_{\mathbb{Q}_p\setminus\mathbb{Z}_p}(x^3+ax-b)$ of roots of the cubic equation \eqref{1.1} in $\mathbb{Z}^*_p,$ $\mathbb{Z}_p\setminus\mathbb{Z}^{*}_p,$ and $\mathbb{Q}_p\setminus\mathbb{Z}_p$. All these results are presented in  Section \ref{NumberRoots}.

\begin{figure}[H]
	\centering
		\includegraphics[width=0.5\textwidth]{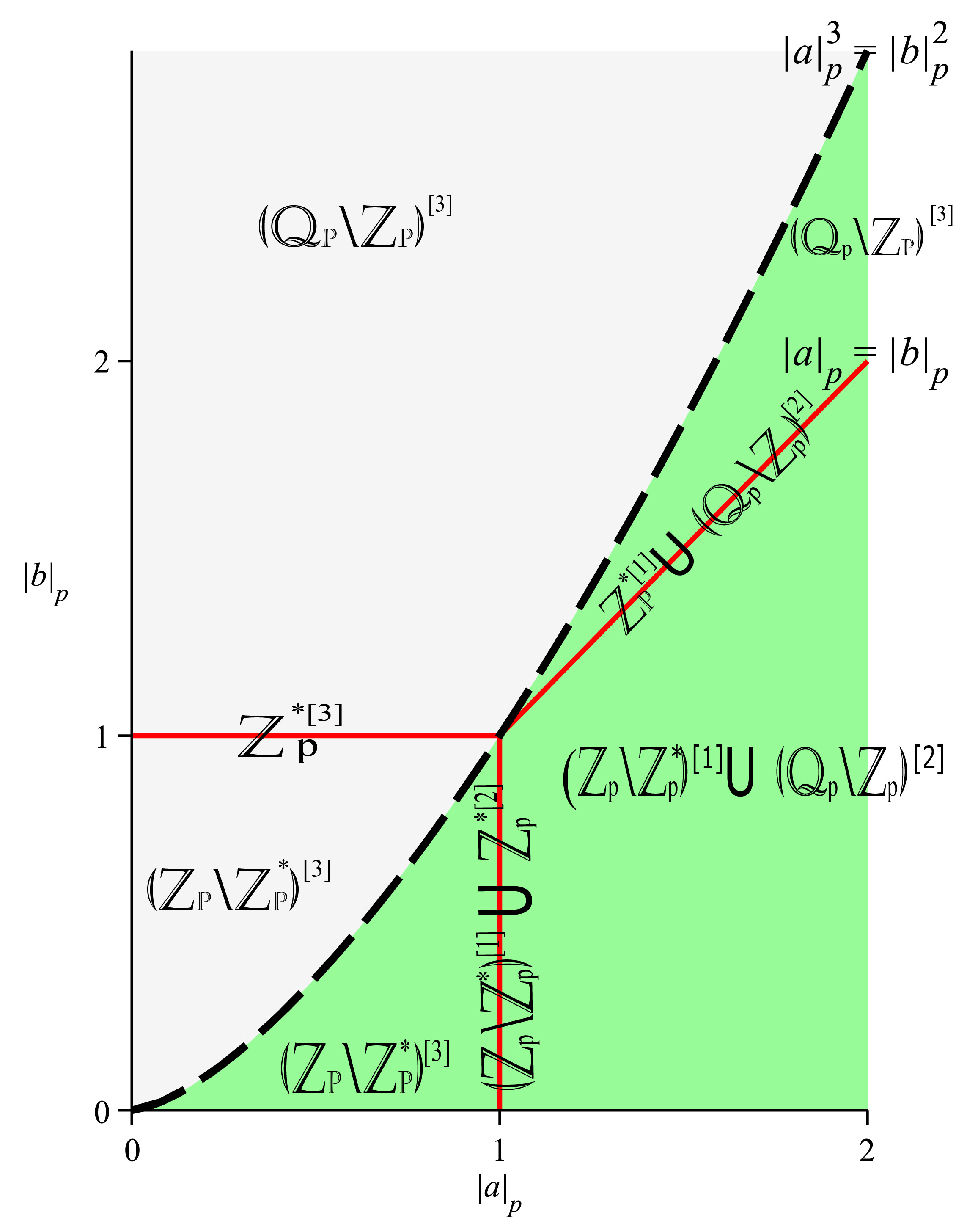}
		\rule{35em}{0.4pt}
	\caption[The description of solutions]{$(a,b) \in \Delta_1$ with $p \equiv  1 \ (mod \ 3)$ and $(a,b) \in \Delta_3$ with $\sqrt{-a}-\exists$}
	\label{figDescription1}
\end{figure}

\begin{figure}[H]
	\centering
		\includegraphics[width=0.5\textwidth]{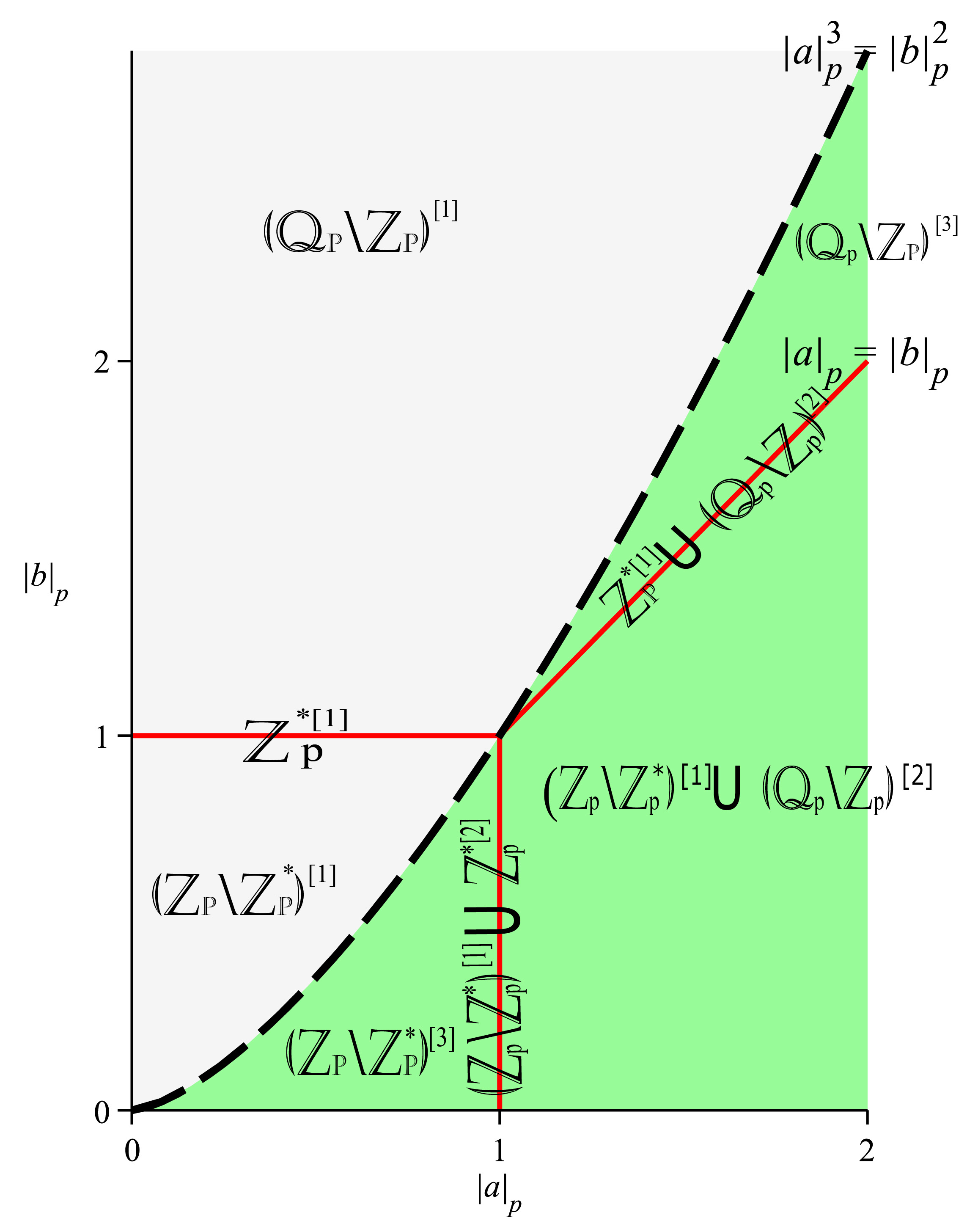}
		\rule{35em}{0.4pt}
	\caption[The description of solutions]{$(a,b) \in \Delta_1$ with $p \equiv 2  \ (mod \ 3)$ and $(a,b) \in \Delta_3$ with $\sqrt{-a}-\exists$}
	\label{figDescription2}
\end{figure}
\begin{figure}[H]
	\centering
		\includegraphics[width=0.5\textwidth]{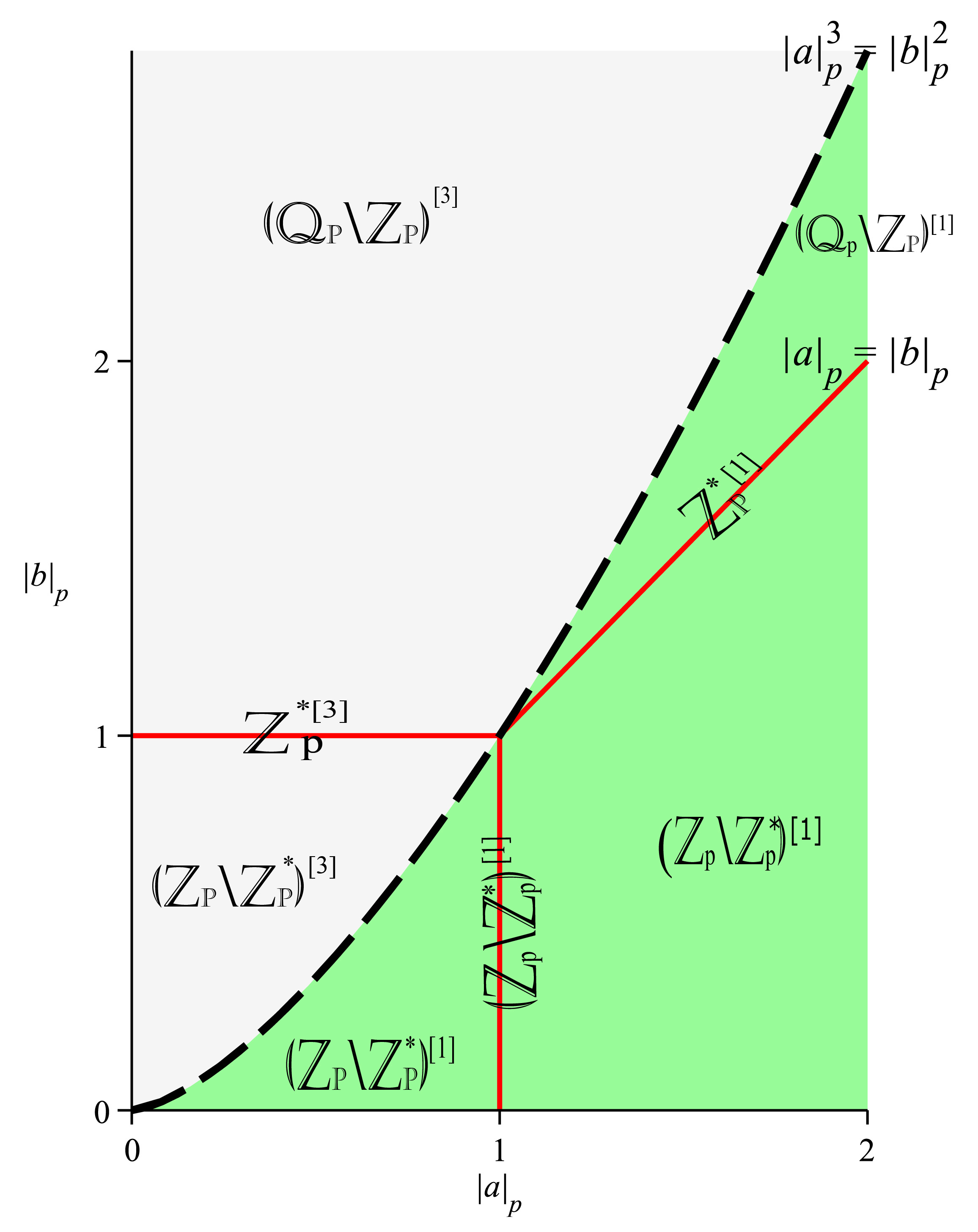}
		\rule{35em}{0.5pt}
	\caption[The description of solutions]{$(a,b) \in \Delta_1$ with $p \equiv  1 \ (mod \ 3)$ and $(a,b) \in \Delta_3$ with $\sqrt{-a}-\not\exists$}
	\label{figDescription3}
\end{figure}
\begin{figure}[H]
	\centering
		\includegraphics[width=0.5\textwidth]{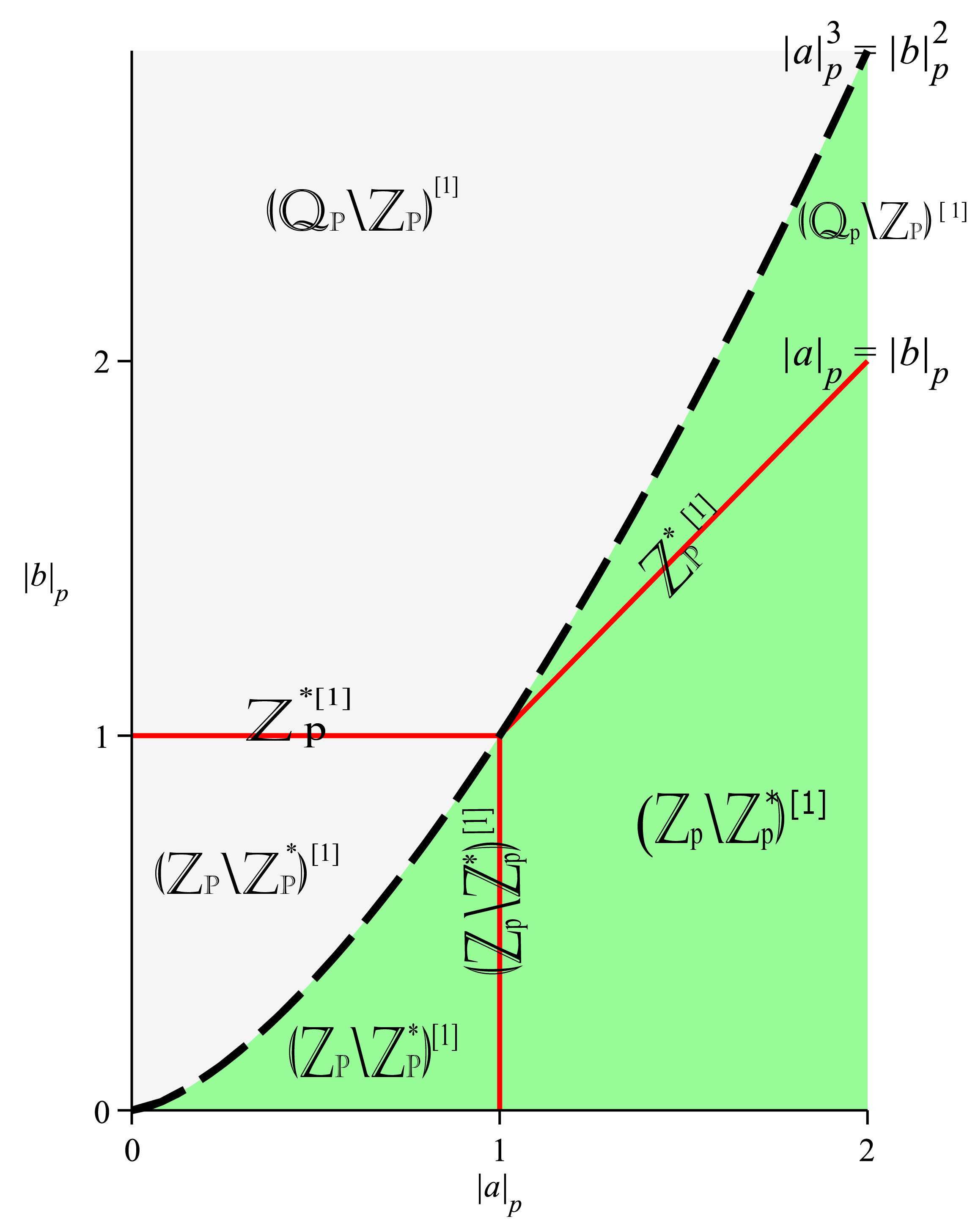}
		\rule{35em}{0.5pt}
	\caption[The description of solutions]{$(a,b) \in \Delta_1$ with $p \equiv 2 \ (mod \ 3)$ and $(a,b) \in \Delta_3$ with $\sqrt{-a}-\not\exists$}
	\label{figDescription4}
\end{figure}
\begin{figure}[H]
	\centering
		\includegraphics[width=0.5\textwidth]{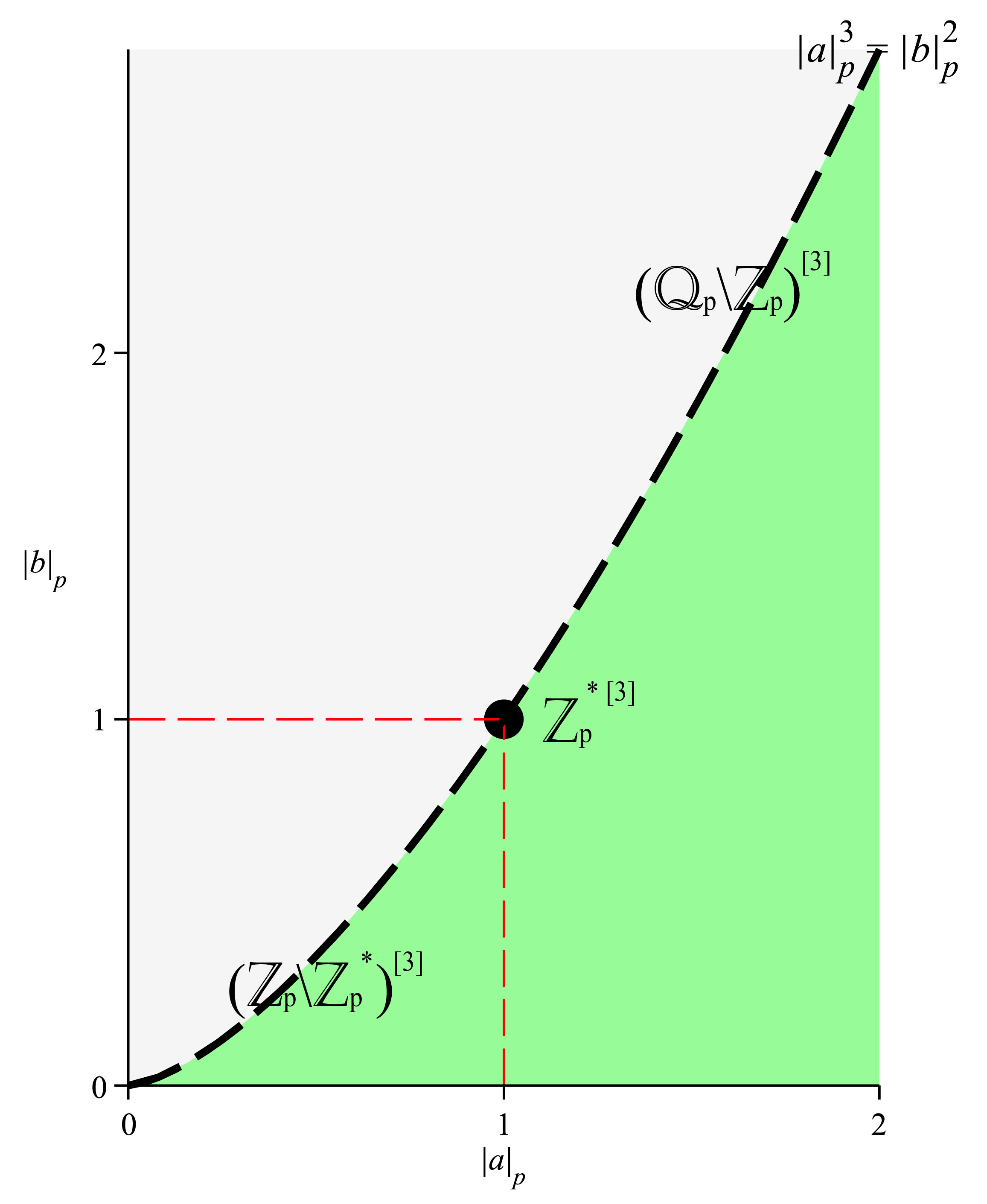}
		\rule{35em}{0.5pt}
	\caption[The description of solutions]{$(a,b) \in \Delta_2$ and $|D|_p=1$, $D_0u_{p-2}^2 \equiv 0$ or  $0 \leq |D|_p < 1$, $\sqrt{D}-\exists$}
	\label{figDescription5}
\end{figure}
\begin{figure}[H]
	\centering
		\includegraphics[width=0.5\textwidth]{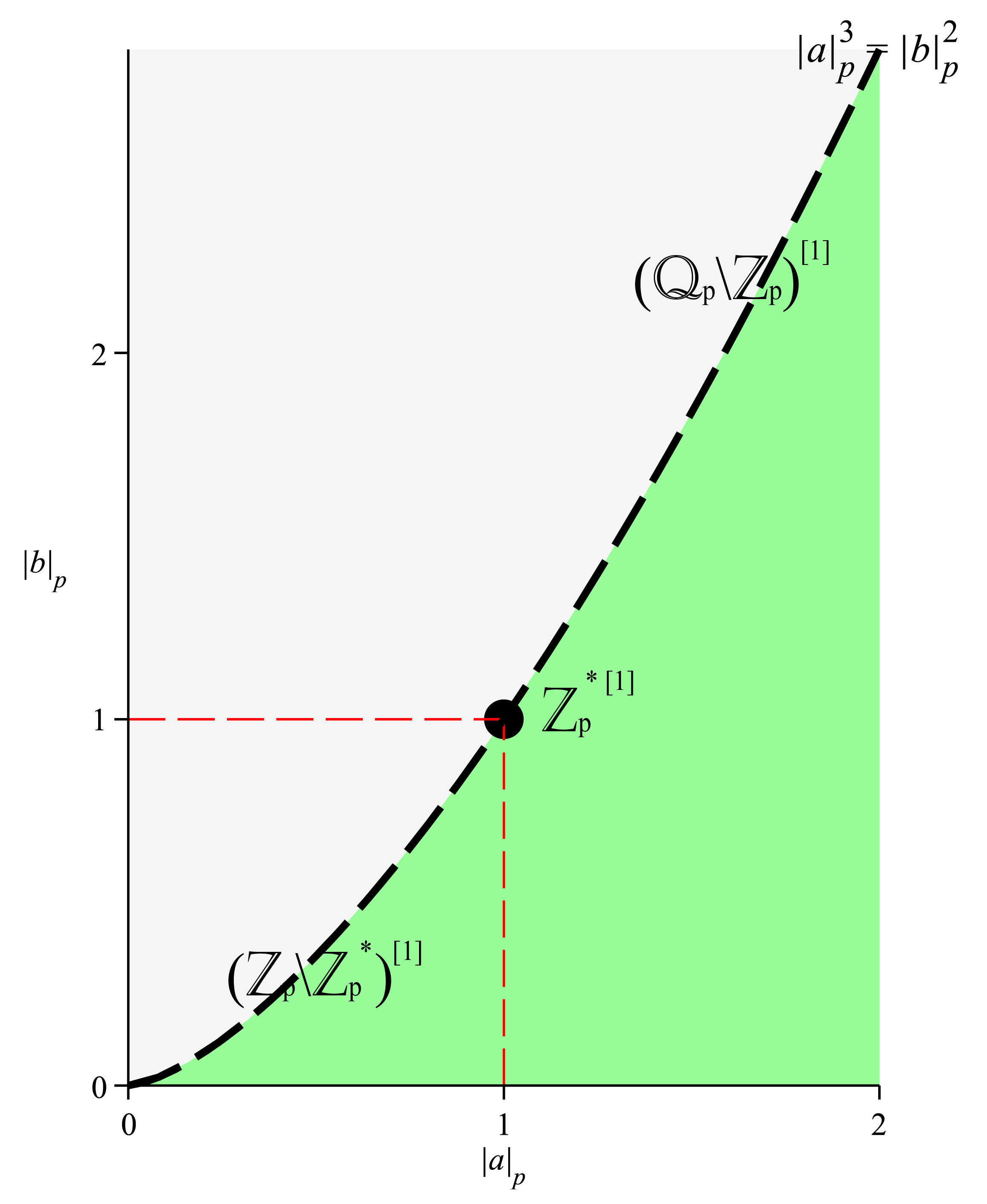}
		\rule{35em}{0.5pt}
	\caption[The description of solutions]{$(a,b) \in \Delta_2$ and $|D|_p=1$ $D_0u_{p-2}^2 \not\equiv 0,9a_0^2$ or  $0 < |D|_p < 1$ $\sqrt{D}-\not\exists$}
	\label{figDescription6}
\end{figure}

\section{The Main Strategies}\label{Strategies}

We know that, by definition, two $p-$adic numbers are closed when their difference is divisible by a high power of $p$. This property enables $p-$adic numbers to encode congruence information in a way that turns out to be powerful tools in the theory of polynomial equation. In fact, Hensel's lifting lemma allows us to lift a simple solution of a polynomial equation over the finite field $\mathbb{F}_p$ up to the unique solution of the same polynomial equation over the ring $\mathbb{Z}_p$ of $p-$adic integer numbers. However, that solution cannot be lifted up to the field $\mathbb{Q}_p$ of $p-$adic numbers. At this point, we are aiming to study the relation between solutions of the polynomial equations over $\mathbb{Q}_p$ and $\mathbb{Z}_p$. We shall show that, indeed, any solution of any cubic equation over $\mathbb{Q}_p$ (or some special domains) can be uniquely determined by a solution of another cubic equation over $\mathbb{Z}^{*}_p$. Consequently, it is enough to study cubic equations over $\mathbb{Z}^{*}_p$.    

Let us consider the cubic equation
\begin{equation}
\label{1}
x^3+ax=b.
\end{equation}
where $a,b\in\mathbb{Q}_p$ and $ab \neq 0$.

Let $\mathbb{K} \subset \mathbb{Z}$ be any subset. We introduce the following set 
$$
\frac{\mathbb{Z}_p^{*}}{p^{\mathbb{K}}}:= \left\{x \in \mathbb{Q}_p : \log_p|x|_p \in \mathbb{K} \right\}. 
$$ 

It is easy to check that 
$$
\frac{\mathbb{Z}_p^{*}}{p^{\mathbb{K}}}=\bigcup\limits_{i\in\mathbb{K}}\mathbb{S}_{p^{i}}(0),
$$
where $\mathbb{S}_{p^{i}}(0)=\{x\in\mathbb{Q}_p: |x|_p=p^{i}\}$ is the sphere with the radius $p^{i}$. 

\begin{proposition}\label{solvableinQ_pandZ_p}
Let $p$ be any prime, $a,b\in\mathbb{Q}_p$ with $ab\neq 0$, and $\mathbb{K} \subset \mathbb{Z}$ be any subset. The cubic equation \eqref{1} is solvable in the set $\ \cfrac{\mathbb{Z}_p^{*}}{p^{\mathbb{K}}} \ $ if and only if there exists a pair $(y^{*},k) \in \mathbb{Z}_p^{*} \times \mathbb{K}$ such that $y^{*}$ is a root of the following cubic equation
\begin{eqnarray}\label{2}
y^3+A_ky=B_k
\end{eqnarray}
where $A_k=ap^{2k}$ and $B_k=bp^{3k}$. Moreover, in this case, a root of the cubic equation \eqref{1} has the form $x = \cfrac{y^{*}}{p^{k}}$.
\end{proposition}
\begin{proof}
Let $x\in \mathbb{Q}_p$ and $|x|_p=p^{k}$. Then $x\in\cfrac{\mathbb{Z}_p^{*}}{p^{\mathbb{K}}}$ is a root of the cubic equation \eqref{1} if and only if $y^{*}=x|x|_p\in \mathbb{Z}_p^{*}$ is a root of the cubic equation \eqref{2}. This completes the proof.
\end{proof}

Here, we list frequently used domains in this paper.

\begin{itemize}
\item[1.] If $\mathbb{K}_1 = \{0\}$ then $\cfrac{\mathbb{Z}_p^{*}}{p^{\mathbb{K}_1}}=\mathbb{Z}^{*}_p$;
\item[2.] If $\mathbb{K}_2 = \mathbb{N}_{-}$ then $\cfrac{\mathbb{Z}_p^{*}}{p^{\mathbb{K}_2}}=\mathbb{Z}_p \setminus \mathbb{Z}^{*}_p$;
\item[3.] If $\mathbb{K}_3 = \mathbb{N}$ then $\cfrac{\mathbb{Z}_p^{*}}{p^{\mathbb{K}_3}}=\mathbb{Q}_p \setminus \mathbb{Z}_p$;
\item[4.] If $\mathbb{K}_4 = \mathbb{Z}$ then $\cfrac{\mathbb{Z}_p^{*}}{p^{\mathbb{K}_4}}=\mathbb{Q}_p$.
\item[5.] If $\mathbb{K}_5 = \mathbb{Z}_{-}$ then $\cfrac{\mathbb{Z}_p^{*}}{p^{\mathbb{K}_5}}=\mathbb{Z}_p$;
\item[6.] If $\mathbb{K}_6 = \mathbb{Z}_{+}$ then $\cfrac{\mathbb{Z}_p^{*}}{p^{\mathbb{K}_6}}=\mathbb{Q}_p \setminus \left(\mathbb{Z}_p \setminus \mathbb{Z}^{*}_p\right)$;
\item[7.] If $\mathbb{K}_7 =\mathbb{Z}\setminus\{0\}$ then $\cfrac{\mathbb{Z}_p^{*}}{p^{\mathbb{K}_7}}=\mathbb{Q}_p\setminus\mathbb{Z}^{*}_p$. 
\end{itemize}

\begin{corollary}\label{Aisets}
Let $p$ be any prime, $a,b\in\mathbb{Q}_p$ with $ab\neq 0$, and $\mathbb{K}_i \subset \mathbb{Z}$ be a subset given as above, $i=\overline{1,7}$. The cubic equation \eqref{1} is solvable in the set $\ \cfrac{\mathbb{Z}_p^{*}}{p^{\mathbb{K}_i}}\ $ if and only if there exists $(y^{*},k) \in \mathbb{Z}_p^{*} \times \mathbb{K}_i$ such that $y^{*}$ is a root of the following cubic equation 
\begin{eqnarray*}
y^3+A_ky=B_k
\end{eqnarray*}
where $A_k=ap^{2k}$ and $B_k=bp^{3k}$. Moreover, in this case, a root of the cubic equation \eqref{1} has the form $x = \cfrac{y^{*}}{p^{k}}$.
\end{corollary}

Consequently, it is enough to study solvability of the cubic
equation \eqref{1} over $\mathbb{Z}_p^{*}$, where $a,b \in \mathbb{Q}_p$ with $ab \neq 0$.

\begin{proposition}[\cite{FMBOMS}]\label{conditionforZ^{*}_p}
Let $p$ be any prime and $a,b\in\mathbb{Q}_p$ with $ab\neq 0$. If the cubic equation \eqref{1} is solvable in $\mathbb{Z}_p^{*}$ then either one of the following conditions holds true:
\begin{itemize}
  \item [$(i)$] $|a|_p=|b|_p\geq 1;$
  \item [$(ii)$] $|b|_p<|a|_p=1;$
  \item [$(iii)$] $|a|_p<|b|_p=1$.
\end{itemize}
\end{proposition}

This proposition gives necessary conditions for solvability of the cubic equation over $\mathbb{Z}_p^{*}$. To get the solvability criteria, we need Hensel's lifting lemma.

\begin{lemma}[Hensel's Lemma, \cite{Bor Shaf}]\label{Hensel}
Let $f(x)$ be polynomial whose the coefficients are $p-$adic
integers. Let $\theta$ be a $p-$adic integer such that for some
$i\geq 0$ we have
$$
f(\theta)\equiv 0 \ (mod \ p^{2i+1}),
$$
$$
f'(\theta)\equiv 0 \ (mod \ p^{i}), \quad f'(\theta)\not\equiv 0 \ (mod \ p^{i+1}).
$$
Then $f(x)$ has a unique $p-$adic integer root $x_0$ which satisfies $x_0\equiv \theta\ (mod \ p^{i+1}).$
\end{lemma}
The following result was implicitly proven in \cite{FMBOMS}.
\begin{proposition}[\cite{FMBOMS}]\label{CriteriainZp*}
Let $p>3$ and $a,b\in\mathbb{Q}_p$ with $ab\neq 0.$ Let $a=\frac{a^{*}}{|a|_p}$ and $b=\frac{b^{*}}{|b|_p}$, where
$a^{*}=a_0+a_1p+a_2p^2+\cdots,$ $b^{*}=b_0+b_1p+b_2p^2+\cdots$. The following statements hold true. 
 \begin{itemize}
    \item [$(i)$] {If $|a|_p<|b|_p=1$ or $|b|_p<|a|_p=1$ or $|a|_p=|b|_p=1$ then the cubic equation \eqref{1} is solvable in $\mathbb{Z}_p^{*}$ if and only if the following congruent equation $x^3+ax\equiv b\ (mod \ p)$ is solvable in $\mathbb{F}_p.$ Moreover, in this case, one has that $\mathbf{N}_{\mathbb{Z}^*_p}(x^3+ax-b)=\mathbf{N}_{\mathbb{F}_p}(x^3+ax-b)$}
    \item [$(ii)$] {If $|a|_p=|b|_p>1$  then the cubic equation \eqref{1} is solvable in $\mathbb{Z}_p^{*}$ if and only if the following congruent equation $a^{*}x\equiv b^{*}\ (mod \ p)$ is solvable in  $\mathbb{F}_p$ Moreover, in this case, one has that $\mathbf{N}_{\mathbb{Z}^*_p}(x^3+ax-b)=\mathbf{N}_{\mathbb{F}_p}(a^{*}x-b^{*})$}
  \end{itemize} 
\end{proposition}

Let us consider the following cubic equation in the finite field $\bbf_p$
\begin{eqnarray}\label{cubiccong}
x^3+\bar{a}x=\bar{b},
\end{eqnarray}
where $\bar{a},\bar{b}\in \bbf_p.$  We assume that $\bar{a}\neq\bar{0}$ and $\bar{b}\neq\bar{0}$. The number ${\mathbf{N}}_{\bbf_p}(x^3+\bar{a}x-\bar{b})$ of roots of this equation was described in the papers \cite{ZHS1}-\cite{ZHS3}.

\begin{proposition}[\cite{ZHS1}-\cite{ZHS3}]\label{CubicinF_p}
Let $p>3$ be a prime number and $\bar{a},\bar{b}\in\bbf_p$ with
$\bar{a}\bar{b}\neq\bar{0}$. Let
$\overline{D}=-4\bar{a}^3-27\bar{b}^2$ and
$u_{n+3}=\bar{b}u_n-\bar{a}u_{n+1}$ for $n\in\bn$ with
$u_1=\bar{0},$ $u_2=-\bar{a},$ $u_3=\bar{b}.$ Then the following
holds:
$$
{\mathbf{N}}_{\bbf_p}(x^3+\bar{a}x-\bar{b})=\left\{
\begin{array}{l}
3 \ \ \  if \ \ \ \overline{D}u_{p-2}^2=\bar{0} \\
0 \ \ \ if \ \ \ \overline{D}u_{p-2}^2=9\bar{a}^2 \\
1 \ \ \ if \ \ \ \overline{D}u_{p-2}^2\neq \bar{0}, 9\bar{a}^2
\end{array}
\right.
$$
\end{proposition}

By means of Proposition \ref{CriteriainZp*} and Hensel's roots formula  \cite{BGW} (see Theorem 3.1.), we can write an explicit formula for the roots of the cubic equation \eqref{1}. 

\begin{proposition}
Let $p>3$ and $a,b\in\mathbb{Q}_p$ with $ab\neq 0.$ Let $a=\frac{a^{*}}{|a|_p}$ and $b=\frac{b^{*}}{|b|_p}$, where
$a^{*}=a_0+a_1p+a_2p^2+\cdots,$ $b^{*}=b_0+b_1p+b_2p^2+\cdots$. Suppose the cubic equation \eqref{1} is solvable in $\mathbb{Z}^{*}_p.$ Let $r_0$ be a root of the following congruent equation
\begin{itemize}
\item[$(i)$] $x^3\equiv b_0 \ (mod \ p)$ whenever $|a|_p<|b|_p=1$
\item[$(ii)$] $x^2+a_0\equiv 0 \ (mod \ p)$ whenever $|b|_p<|a|_p=1$
\item[$(iii)$] $x^3+a_0x\equiv b_0 \ (mod \ p)$ whenever $|a|_p=|b|_p=1$
\item[$(iv)$] $a_0x\equiv b_0 \ (mod \ p)$ whenever $|a|_p=|b|_p>1$ 
\end{itemize}
Then Hensel's lifting root of the cubic equation \eqref{1} in the domain $\mathbb{Z}^{*}_p$ has the following form
$$
r=r_0-\frac{c_0}{c_1}\sum\limits_{k=0}^{\infty}\left[\sum\limits_{j=0}^{k}\frac{(-1)^{k-j}c_2^j}{2k-j+1}\binom{k}{j}\binom{3k-j}{k}\left(\frac{c_0c_3}{c_1}\right)^{k-j}\right]\left(\frac{c_0}{c^2_1}\right)^k
$$
where 
\begin{itemize}
\item[$(i)$] {$c_0=r_0^3+ar_0-b,$ $c_1=3r_0^2+a,$ $c_2=3r_0,$ $c_3=1$ whenever $|a|_p<|b|_p=1$ or $|b|_p<|a|_p=1$ or $|a|_p=|b|_p=1$}
\item[$(ii)$] {$c_0=p^\gamma r_0^3+a^{*}r_0-b^{*},$ $c_1=3p^{\gamma}r_0^2+a^{*},$ $c_2=3p^{\gamma}r_0,$ $c_3=p^{\gamma}$ whenever $|a|_p=|b|_p=p^\gamma>1$ }
\end{itemize} 
\end{proposition}

\section{The Solvability Criteria}\label{Criteria}

We present a solvability criterion of the cubic equation  
\begin{eqnarray}\label{3.1}
x^3+ax=b
\end{eqnarray}
over $\mathbb{A}$  where
$$\mathbb{A}\in\left\{\mathbb{Z}_p^{*}, \ \ \mathbb{Z}_p\setminus\mathbb{Z}_p^{*}, \ \ \mathbb{Q}_p\setminus\mathbb{Z}_p\right\}.$$ 

We have that $a=\cfrac{a^{*}}{|a|_p}$ and $b=\cfrac{b^{*}}{|b|_p}$, where
\begin{eqnarray*}
a^{*}&=&a_0+a_1p+a_2p^2+\cdots\\
b^{*}&=&b_0+b_1p+b_2p^2+\cdots
\end{eqnarray*}
where $a_0,b_0\in\{1,2,\cdots p-1\}$ and $a_i,b_i\in\{0,1,2,\cdots p-1\}$ for any $i\in\mathbb{N}.$

We set $D_0=-4a_0^3-27b_0^2$ and $u_{n+3}=b_0u_n-a_0u_{n+1}$ with $u_1=0,$ $u_2=-a_0,$ and $u_3=b_0$ for $n=\overline{1,p-3}$. 

\begin{theorem}\label{CriteriaforSolutions}
The cubic equation \eqref{3.1} is 
\begin{itemize}
\item[$I$] {Solvable in $\mathbb{Z}^{*}_p$ iff either one of the following conditions holds true:
  \begin{itemize}
    \item [$I.1.$] $|a|_p<|b|_p=1$ and $\sqrt[3]{b}-\exists;$
    \item [$I.2.$] $|b|_p<|a|_p=1$ and $\sqrt{-a}-\exists;$
    \item [$I.3.$] $|a|_p=|b|_p=1$ and $D_0u_{p-2}^2\not\equiv 9a_0^{2} \ (mod \ p);$
    \item [$I.4.$] $|a|_p=|b|_p>1.$
  \end{itemize}}
\item[II] {Solvable in $\mathbb{Z}_p\setminus\mathbb{Z}^{*}_p$ iff either one of the following conditions holds true:
\begin{itemize}
  \item [$II.1.$] $|a|_p^3<|b|_p^2 < 1$ and \ $\sqrt[3]{b}-\exists;$
  \item [$II.2.$] $|a|_p^3=|b|_p^2 < 1$ and $D_0u_{p-2}^2\not\equiv 9a_0^{2} \ (mod \ p);$
  \item [$II.3.$] $|a|_p^3>|b|_p^2$ and $|a|_p > |b|_p.$
\end{itemize}}
\item[III] {Solvable in $\mathbb{Q}_p\setminus\mathbb{Z}_p$ iff either one of the following conditions holds true:
\begin{itemize}
  \item [$III.1.$] $|a|_p^3<|b|_p^2,$ \ $|b|_p>1,$ and \ $\sqrt[3]{b}-\exists;$
  \item [$III.2.$] $|a|_p^3=|b|_p^2 > 1$ and $D_0u_{p-2}^2\not\equiv 9a_0^{2} \ (mod \ p);$
  \item [$III.3.$] $|a|_p^3>|b|_p^2$ and
  {\begin{itemize}
  \item[$III.3. (i)$] $|a|_p<|b|_p$ \text{or}
  \item[$III.3. (ii)$] $|a|_p \geq |b|_p,$ \ $|a|_p > 1,$  and \ $\sqrt{-a}-\exists;$
  \end{itemize}}
\end{itemize}}
\end{itemize}
\end{theorem}
\begin{proof}
The proof of the solvability criterion in $\mathbb{Z}^{*}_p$ was presented in \cite{FMBOMS}. In order to prove the solvability criteria in $\mathbb{Z}_p\setminus\mathbb{Z}^{*}_p$ or in $\mathbb{Q}_p\setminus\mathbb{Z}_p$, we have to apply Corollary \ref{Aisets} and Proposition \ref{CriteriainZp*}. Since the proofs in the cases $\mathbb{Z}_p\setminus\mathbb{Z}^{*}_p$ and $\mathbb{Q}_p\setminus\mathbb{Z}_p$ are similar, hence, we just provide the proof of the solvability criterion in the case $\mathbb{Z}_p\setminus\mathbb{Z}^{*}_p$.

Let $x\in\mathbb{Q}_p$ be a nonzero $p-$adic number and $|x|_p = p^k$ where $k \in \mathbb{Z}$. Due to Corollary \ref{Aisets}, $x$ is a solution of the depressed cubic equation  \eqref{3.1} in $\mathbb{Z}_p\setminus\mathbb{Z}^{*}_p$ iff $y=p^{k}x$ is a solution of the following cubic equation
\begin{eqnarray}\label{2.2}
y^3+A_ky=B_k
\end{eqnarray}
in $\mathbb{Z}_p^{*}$ for some negative integer $k \in \mathbb{N}_{-}$,  where $A_k = ap^{2k}$ and $B_k = bp^{3k}$. 

It is clear that $A_k^{*}= a^{*}, \ B_k^{*} = b^{*}$ and $|A_k|_p = p^{-2k}|a|_p, \ |B_k|_p = p^{-3k}|b|_p.$ We know that, the equation \eqref{2.2} is solvable in $\mathbb{Z}_p^{*}$ iff either one of the following conditions holds true:
\begin{itemize}
    \item [$(i)$] $|A_k|_p<|B_k|_p=1$ and $\sqrt[3]{B_k}-\exists;$
    \item [$(ii)$]   $|B_k|_p<|A_k|_p=1$ and $\sqrt{-A_k}-\exists;$
    \item [$(iii)$] $|A_k|_p=|B_k|_p=1$ and $D_0u_{p-2}^2\not\equiv 9a_0^{2} \ (mod \ p);$
    \item [$(iv)$] $|A_k|_p=|B_k|_p>1.$
\end{itemize}
where $A_k=\cfrac{a^{*}}{p^{-2k}|a|_p}, \ B_k=\cfrac{b^{*}}{p^{-3k}|b|_p}$ with
\begin{eqnarray*}
a^{*}&=&a_0+a_1p+a_2p^2+\cdots\\
b^{*}&=&b_0+b_1p+b_2p^2+\cdots
\end{eqnarray*}
$a_0,b_0\in\{1,2,\cdots p-1\}$ and $a_i,b_i\in\{0,1,2,\cdots p-1\}$ for any $i\in\mathbb{N}.$
 
We want to describe all $p-$adic numbers $a,b \in \mathbb{Q}_p$ for which at least one of the conditions $(i)-(iv)$ should be satisfied for some $k \in \mathbb{N}_{-}$.

$1.$ Let us consider the condition $(i)$: Suppose that $|A_k|_p<|B_k|_p=1$ and there exists $\sqrt[3]{B_k}$ or equivalently $b_0^{\frac{p-1}{(3,p-1)}}\equiv 1 \ (mod \ p).$ Since $|B_k|_p=p^{-3k}|b|_p,$ we get from $|B_k|_p=1$ that $3k=\log_p|b|_p.$ The last equation has a negative integer solution w.r.t $k$ if and only if $\log_p|b|_p$ is divisible by 3 and $|b|_p < 1$. Therefore, if $k=\frac{\log_p|b|_p}{3}$ then it follows from $|A_k|_p=p^{2k}|a|_p<1$ that $|a|_p^3<|b|_p^2.$ Consequently, if $|a|_p^3<|b|_p^2 < 1,$ $b_0^{\frac{p-1}{(3,p-1)}}\equiv 1 \ (mod \ p)$ and $3|\log_p|b|_p$ then the equation \eqref{2.2} has a solution in $\mathbb{Z}_p^{*}$ for $k=\frac{\log_p|b|_p}{3}$.

$2.$ Let us consider the condition $(iii)$: Assume that $|A_k|_p=|B_k|_p=1$ and $D_0u_{p-2}^2\not\equiv 9a_0^{2} \ (mod \ p).$ Since $|A_k|_p=p^{-2k}|a|_p,$ $|B_k|_p=p^{-3k}|b|_p$ we have that $2k=\log_p|a|_p$ and $3k=\log_p|b|_p.$ The last two equations have a negative integer solution w.r.t $k$ if and only if $2\mid\log_p|a|_p$, $3\mid\log_p|b|_p$, and $|a|_p< 1$, $|b|_p< 1.$ Therefore, if $k=\frac{\log_p|a|_p}{2}=\frac{\log_p|b|_p}{3}$ then we have that $|a|_p^3=|b|_p^2.$ Consequently, if $|a|_p^3=|b|_p^2< 1$ and $D_0u_{p-2}^2\not\equiv 9a_0^{2} \ (mod \ p)$ then the equation \eqref{2.2} has a solution in $\mathbb{Z}_p^{*}$ for $k=\frac{\log_p|a|_p}{2}=\frac{\log_p|b|_p}{3}$.

$3.$ Let us consider the condition $(ii)$ and $(iv)$: We are going to consider two cases: Case-$(ii)$ $|B_k|_p<|A_k|_p=1,$ there exists $\sqrt{-A_k}$ or equivalently $(-a_0)^{\frac{p-1}{2}}\equiv 1 \ (mod \ p);$ Case-$(iv)$ $|A_k|_p=|B_k|_p>1.$ Let us study every case.

Let $|B_k|_p<|A_k|_p=1$ and $(-a_0)^{\frac{p-1}{2}}\equiv 1 \ (mod \ p).$ Since $|A_k|_p=p^{-2k}|a|_p$ we get from $|A_k|_p=1$ that $2k=\log_p|a|_p.$ The last equation has a negative integer solution w.r.t $k$ if and only if $2\mid\log_p|a|_p$ and $|a|_p<1$. Therefore, if $k=\frac{\log_p|a|_p}{2}$ then it follows from $|B_k|_p=p^{-3k}|b|_p<1$ that $|a|_p^3>|b|_p^2.$ Consequently, if $|b|_p^2<|a|_p^3<1,$ $2\mid\log_p|a|_p$, and $(-a_0)^{\frac{p-1}{2}}\equiv 1 \ (mod \ p)$ then the equation \eqref{2.2} has a solution in $\mathbb{Z}_p^{*}$ for $k=\frac{\log_p|a|_p}{2}$.

Let $|A_k|_p=|B_k|_p>1.$ Since $|A_k|_p=p^{-2k}|a|_p,$ $|B_k|_p=p^{-3k}|b|_p$ we obtain from $|A_k|_p=|B_k|_p$ that $k=\log_p|b|_p-\log_p|a|_p.$ Hence, $k$ is a negative integer if and only if $|a|_p>|b|_p.$ Therefore, if $k=\log_p|b|_p-\log_p|a|_p$ then it follows form $|A_k|_p=|B_k|_p>1$ that $|a|_p^3>|b|_p^2.$ Consequently, if $|a|_p^3>|b|_p^2$ and $|a|_p>|b|_p$ then the equation \eqref{2.2} has a solution in $\mathbb{Z}_p^{*}$ for $k=\log_p|b|_p-\log_p|a|_p$.

It is worth mentioning that if $p-$adic numbers $a,b\in\mathbb{Q}_p$ satisfy the conditions  $|b|_p^2<|a|_p^3<1,$ $2\mid\log_p|a|_p$, and $(-a_0)^{\frac{p-1}{2}}\equiv 1 \ (mod \ p)$ then they satisfy the conditions $|a|_p^3>|b|_p^2$ and $|a|_p>|b|_p$ as well. Consequently, regardless of whether $p-$adic numbers $a,b\in\mathbb{Q}_p$ satisfy the conditions $|b|_p^2<|a|_p^3<1,$ $2\mid\log_p|a|_p$, and $(-a_0)^{\frac{p-1}{2}}\equiv 1 \ (mod \ p)$ or not, the equation \eqref{2.2} has a solution in $\mathbb{Z}_p^{*}$ if $|a|_p^3>|b|_p^2$ and $|a|_p>|b|_p$. Moreover, if $|b|_p^2<|a|_p^3<1,$ $2\mid\log_p|a|_p$, and $(-a_0)^{\frac{p-1}{2}}\equiv 1 \ (mod \ p)$ then the equation \eqref{2.2} has at least two distinct solutions in $\mathbb{Z}_p^{*}$ for two distinct negative integers $k=\frac{\log_p|a|_p}{2}$ and $k=\log_p|b|_p-\log_p|a|_p,$ otherwise the equation \eqref{2.2} has at least one solution in $\mathbb{Z}_p^{*}$ for $k=\log_p|b|_p-\log_p|a|_p$. This completes the proof.
\end{proof}

\begin{remark}
Let us turn back to the first problem which was posted in the Section \ref{ProblemsandResults}. By means of Theorem \ref{CriteriaforSolutions} and Theorem 3.2 of the paper \cite{FMBOMS}, we can say that the solvability criteria problems for the domains $\mathbb{Z}_p^{*}, \ \mathbb{Z}_p\setminus\mathbb{Z}_p^{*}, \ \mathbb{Z}_p, \ \mathbb{Q}_p\setminus\mathbb{Z}_p, \ \mathbb{Q}_p$ are solved. Since $\mathbb{Q}_p\setminus\mathbb{Z}_p^{*}= (\mathbb{Z}_p\setminus\mathbb{Z}_p^{*})\cup (\mathbb{Q}_p\setminus\mathbb{Z}_p)$ and $\mathbb{Q}_p\setminus\left(\mathbb{Z}_p\setminus\mathbb{Z}_p^{*}\right)=\mathbb{Z}_p^{*}\cup (\mathbb{Q}_p\setminus\mathbb{Z}_p)$, we can also easily derive the solvability criteria for domains  $\mathbb{Q}_p\setminus\mathbb{Z}_p^{*}, \ \mathbb{Q}_p\setminus\left(\mathbb{Z}_p\setminus\mathbb{Z}_p^{*}\right)$ from Theorem \ref{CriteriaforSolutions}. Consequently, Theorem \ref{CriteriaforSolutions} and Theorem 3.2 of the paper \cite{FMBOMS} provide a complete answer for the solvability criteria problems.
\end{remark}

\section{The Number of Roots}\label{NumberRoots}

\begin{theorem}\label{Numbers}
Let $(a,b)\in \Delta.$ Then the following statements hold true:
\begin{equation*}
\mathbf{N}_{\mathbb{Z}_p^{*}}(x^3+ax-b) = \left\{
\begin{array}{l}
3, \ \ \ |a|_p<|b|_p=1, \ p\equiv 1\ (mod \ 3), \ \sqrt[3]{b}-\exists   \\
3, \ \ \ |a|_p=|b|_p=1, \ 0\leq|D|_p<1, \sqrt{D}-\exists \\
3, \ \ \ |a|_p=|b|_p=1, \ |D|_p=1, \ D_0u_{p-2}^2\equiv 0\ (mod \ p)\\
2, \ \ \ |b|_p<|a|_p=1, \ \sqrt{-a}-\exists\\
1, \ \ \ |a|_p<|b|_p=1, \ p\equiv 2\ (mod \ 3), \ \sqrt[3]{b}-\exists \\
1, \ \ \ |a|_p=|b|_p=1, \ 0<|D|_p<1, \ \sqrt{D}-\not\exists \\
1, \ \ \ |a|_p=|b|_p=1, \ |D|_p=1, \ D_0u_{p-2}^2\not\equiv 0,9a_0^{2}\ (mod \ p)\\
1, \ \ \ |a|_p=|b|_p>1
\end{array} \right.
\end{equation*}

\begin{equation*}
\mathbf{N}_{\mathbb{Z}_p\setminus\mathbb{Z}_p^{*}}(x^3+ax-b) = \left\{
\begin{array}{l}
3, \ \ \ |a|_p^3<|b|_p^2<1, \ p\equiv 1 \ (mod \ 3), \ \sqrt[3]{b}-\exists  \\
3, \ \ \ |a|_p^3=|b|_p^2<1, \ 0\leq|D|_p<1, \ \sqrt{D}-\exists \\
3, \ \ \ |a|_p^3=|b|_p^2<1, \ |D|_p=1, \ D_0u_{p-2}^2\equiv 0\ (mod \ p)\\
3, \ \ \ |b|_p^2<|a|_p^3<1, \ \sqrt{-a}-\exists \\
1, \ \ \ |a|_p^3<|b|_p^2<1, \ p\equiv 2\ (mod \ 3), \ \sqrt[3]{b}-\exists \\
1, \ \ \ |a|_p^3=|b|_p^2<1, \ 0<|D|_p<1, \ \sqrt{D}-\not\exists \\
1, \ \ \ |a|_p^3=|b|_p^2<1, \ |D|_p=1, \ D_0u_{p-2}^2\not\equiv 0,9a_0^{2}\ (mod \ p)\\
1, \ \ \ |b|_p^2<|a|_p^3<1, \ \sqrt{-a}- \not\exists \\
1, \ \ \ |b|_p^2<|a|_p^3, \ \ |b|_p<|a|_p, \ \ |a|_p\geq1.
\end{array} \right.
\end{equation*}

\begin{equation*}
\mathbf{N}_{\mathbb{Q}_p\setminus\mathbb{Z}_p}(x^3+ax-b) = \left\{
\begin{array}{l}
3, \ \ \ |a|_p^3<|b|_p^2, \ |b|_p>1, \ p\equiv 1 \ (mod \ 3), \ \sqrt[3]{b}-\exists \\
3, \ \ \ |a|_p^3=|b|_p^2>1, \ 0\leq|D|_p<1, \ \sqrt{D}-\exists \\
3, \ \ \ |a|_p^3=|b|_p^2>1, \ |D|_p=1, \ D_0u_{p-2}^2\equiv 0\ (mod \ p)\\
3, \ \ \ |a|_p^3>|b|_p^2, \ |a|_p<|b|_p, \ \sqrt{-a}-\exists \\
2, \ \ \ |a|_p^3>|b|_p^2, \ |a|_p\geq|b|_p, \  |a|_p>1, \ \sqrt{-a}-\exists \\
1, \ \ \ |a|_p^3<|b|_p^2, \ |b|_p>1, \ p\equiv 2\ (mod \ 3), \ \sqrt[3]{b}-\exists \\
1, \ \ \ |a|_p^3=|b|_p^2>1, \ 0<|D|_p<1, \ \sqrt{D}-\not\exists \\
1, \ \ \ |a|_p^3=|b|_p^2>1, \ |D|_p=1, \ D_0u_{p-2}^2\not\equiv 0,9a_0^{2}\ (mod \ p)\\
1, \ \ \ |a|_p^3>|b|_p^2, \ |a|_p<|b|_p, \ \sqrt{-a}-\not\exists
\end{array} \right.
\end{equation*}
\begin{equation*}
\mathbf{N}_{\mathbb{Q}_p}(x^3+ax-b) = \left\{
\begin{array}{l}
3, \ \ \ |a|_p^3<|b|_p^2, \ p\equiv 1 \ (mod \ 3), \ \sqrt[3]{b}-\exists \\
3, \ \ \ |a|_p^3=|b|_p^2, \ 0\leq|D|_p<1, \ \sqrt{D}-\exists \\
3, \ \ \ |a|_p^3=|b|_p^2, \ |D|_p=1, \ D_0u_{p-2}^2\equiv 0\ (mod \ p)\\
3, \ \ \ |a|_p^3>|b|_p^2, \ \sqrt{-a}-\exists \\
1, \ \ \ |a|_p^3<|b|_p^2, \ p\equiv 2\ (mod \ 3), \ \sqrt[3]{b}-\exists \\
1, \ \ \ |a|_p^3=|b|_p^2, \ 0<|D|_p<1, \ \sqrt{D}-\not\exists \\
1, \ \ \ |a|_p^3=|b|_p^2, \ |D|_p=1, \ D_0u_{p-2}^2\not\equiv 0,9a_0^{2}\ (mod \ p)\\
1, \ \ \ |a|_p^3>|b|_p^2, \ \sqrt{-a}-\not\exists
\end{array} \right.
\end{equation*}
\end{theorem}
\begin{proof}
The numbers $\mathbf{N}_{\mathbb{Z}_p^{*}}(x^3+ax-b)$ and $\mathbf{N}_{\mathbb{Q}_p}(x^3+ax-b)$ of roots of the cubic equation \eqref{3.1} were already studied in \cite{FMBOMS}. Therefore, we just study $\mathbf{N}_{\mathbb{Z}_p\setminus\mathbb{Z}_p^{*}}(x^3+ax-b)$. The case $\mathbf{N}_{\mathbb{Q}_p\setminus\mathbb{Z}_p}(x^3+ax-b)$ is similar to the case $\mathbf{N}_{\mathbb{Z}_p\setminus\mathbb{Z}_p^{*}}(x^3+ax-b)$ 

The cubic equation \eqref{3.1} is solvable in $\mathbb{Z}_p\setminus\mathbb{Z}^{*}_p$ if and only if either one of conditions of Theorem \ref{CriteriaforSolutions} holds true. We want to find the number of solutions in every case.

Case 1: $|a|_p^3<|b|_p^2<1$ and $\sqrt[3]{b}-\exists$. In this case, we showed that the number of  solutions of the equation \eqref{3.1} in the domain $\mathbb{Z}_p\setminus\mathbb{Z}^{*}_p$ is the same as the number of solutions of the following equation in the domain $\mathbb{Z}_p^{*}$
\begin{eqnarray}\label{II.1}
y^3+a\sqrt[3]{|b|_p^2}y=b^{*}.
\end{eqnarray}

Then it is clear that $\left|a\sqrt[3]{|b|_p^2}\right|_p<\left|b^{*}\right|_p=1$ and $b_0^\frac{p-1}{(3,p-1)} \equiv 1 \ (mod \ p)$. As we already proved that if $p\equiv 1 \ (mod \ 3)$ then the equation \eqref{II.1} has 3 distinct solutions in $\mathbb{Z}_p^{*}$ and if $p\equiv 2 \ (mod \ 3)$ then the equation \eqref{II.1} has a unique solution in $\mathbb{Z}_p^{*}$.

Consequently, if $|a|_p^3<|b|_p^2 < 1,$ $p\equiv 1 \ (mod \ 3),$ and $\sqrt[3]{b}-\exists$ then the depressed cubic equation \eqref{3.1} has 3 distinct solutions in $\mathbb{Z}_p\setminus\mathbb{Z}^*_p$ and if $|a|_p^3<|b|_p^2 < 1,$ $p\equiv 2 \ (mod \ 3)$, and $\sqrt[3]{b}-\exists$ then the depressed cubic equation \eqref{3.1} has a unique solution in $\mathbb{Z}_p\setminus\mathbb{Z}^*_p.$

Case 2: $|a|_p^3=|b|_p^2 < 1$ and $D_0u_{p-2}^2\not\equiv 9a_0^{2} \ (mod \ p).$ In this case, we showed that the number of  solutions of the cubic equation \eqref{3.1} in the domain $\mathbb{Z}_p\setminus\mathbb{Z}^*_p$ is the same as the number of solutions of the following equation in the domain $\mathbb{Z}_p^{*}$
\begin{eqnarray}\label{II.2}
y^3+a^{*}y=b^{*}.
\end{eqnarray}

It is clear that $|a^{*}|_p=|b^{*}|_p=1$ and $D_0u_{p-2}^2\not\equiv 9a_0^{2} \ (mod \ p).$ 

Consequently, under the constrain of Case 2, the cubic equation \eqref{3.1} has 3 solutions if and only if one of the following conditions holds true: (i) $|a|_p^3=|b|_p^2 < 1,$ \ $0\leq|D|_p<1,$ and $\sqrt{D}-\exists$ or (ii) $|a|_p^3=|b|_p^2 < 1,$ $|D|_p=1,$ and $D_0u_{p-2}^2\equiv 0 \ (mod \ p)$. Moreover, under the constrain of Case 2, the cubic equation \eqref{3.1} has a unique solution if and only if one of the following conditions holds true: (i) $|a|_p^3=|b|_p^2 < 1,$ \ $0<|D|_p<1,$ and $\sqrt{D}-\not\exists$ or (ii) $|a|_p^3=|b|_p^2 < 1$, \ $|D|_p=1,$ and $D_0u_{p-2}^2\not\equiv 0,9a_0^{2} \ (mod \ p).$

Case 3: $|b|_p^2<|a|_p^3$ and $|b|_p < |a|_p.$ Let us define the following sets
\begin{eqnarray*}
\nabla&=&\left\{(a,b)\in \mathbb{Q}_p^2: \ |b|_p^2<|a|_p^3, \ |b|_p < |a|_p \right\},\\
\nabla_1&=&\left\{(a,b)\in \mathbb{Q}_p^2: |b|_p^2<|a|_p^3 < 1, \ \sqrt{-a}-\exists\right\},\\
\nabla_2&=&\left\{(a,b)\in \mathbb{Q}_p^2: |b|_p^2<|a|_p^3 < 1, \ \sqrt{-a}-\not\exists \right\},\\
\nabla_3&=&\{(a,b)\in \mathbb{Q}_p^2: \ |b|_p^2<|a|_p^3, \ |b|_p < |a|_p, \ |a|_p\geq1\}.
\end{eqnarray*}
One can easily check that
$$
\nabla=\nabla_1\cup\nabla_2\cup\nabla_3\\
$$

In the domain $\nabla_1$, we showed that the number of  solutions of the depressed cubic equation \eqref{3.1} in the domain $\mathbb{Z}_p\setminus\mathbb{Z}^*_p$ is the same as the total number of solutions of the following two equations in the domain $\mathbb{Z}_p^{*}:$
\begin{eqnarray}
\label{II.3.1}y^3+a^{*}y=b\sqrt{|a|_p^3}, \\
\label{II.3.2}z^3+a\left|\frac{b}{a}\right|_p^2z=b\left|\frac{b}{a}\right|_p^3,
\end{eqnarray}
It is clear that $\left|b\sqrt{|a|_p^3}\right|_p<|a^{*}|_p=1$, \ $\sqrt{-a^{*}}-\exists$ and $\left|a\left|\frac{b}{a}\right|_p^2\right|_p=\left|b\left|\frac{b}{a}\right|_p^3\right|_p>1$.

In this case, as we already proved, the equation \eqref{II.3.1} has 2 distinct solutions in $\mathbb{Z}_p^{*}$ and the equation \eqref{II.3.2} has a unique solution in $\mathbb{Z}_p^{*}$. Consequently, the cubic equation \eqref{3.1} has 3 solutions in $\mathbb{Z}_p\setminus\mathbb{Z}^*_p$.

In the case $\nabla_2\cup\nabla_3$, we showed that the number of  solutions of the depressed cubic equation \eqref{3.1} in the domain $\mathbb{Z}_p\setminus\mathbb{Z}^*_p$ is the same as the number of solutions of the equation \eqref{II.3.2} in the domain $\mathbb{Z}_p^{*}.$ We know that the equation \eqref{II.3.2} has a unique solution in $\mathbb{Z}_p^{*}$. Consequently, the depressed cubic equation \eqref{3.1} has a unique solution in $\mathbb{Z}_p\setminus\mathbb{Z}^*_p$. This completes the proof. 
\end{proof}

\begin{remark}
Let us turn back to the second problem which was posted in the Section \ref{ProblemsandResults}. By means of Theorem \ref{Numbers} and Theorem 4.1 of the paper \cite{FMBOMS}, we can say that the problems concerning the number of roots of the cubic equation \eqref{3.1} in the domains $\mathbb{Z}_p^{*}, \ \mathbb{Z}_p\setminus\mathbb{Z}_p^{*}, \ \mathbb{Z}_p, \ \mathbb{Q}_p\setminus\mathbb{Z}_p, \ \mathbb{Q}_p$ are solved. Since $\mathbb{Q}_p\setminus\mathbb{Z}_p^{*}= (\mathbb{Z}_p\setminus\mathbb{Z}_p^{*})\cup (\mathbb{Q}_p\setminus\mathbb{Z}_p)$ and $\mathbb{Q}_p\setminus\left(\mathbb{Z}_p\setminus\mathbb{Z}_p^{*}\right)=\mathbb{Z}_p^{*}\cup (\mathbb{Q}_p\setminus\mathbb{Z}_p)$, we can also easily derive the number of roots of the cubic equation \eqref{3.1} in domains  $\mathbb{Q}_p\setminus\mathbb{Z}_p^{*}, \ \mathbb{Q}_p\setminus\left(\mathbb{Z}_p\setminus\mathbb{Z}_p^{*}\right)$ from Theorem \ref{Numbers}. Consequently, Theorem \ref{Numbers} and Theorem 4.1 of the paper \cite{FMBOMS} provide a complete answer for the second problem. The graphical illustration of Theorem \ref{Numbers} is presented in Figs.\ref{figDescription1}-\ref{figDescription6}. This answers the third problem of Section \ref{ProblemsandResults}.
\end{remark}

\section*{Acknowledgments}

The authors acknowledge the IIUM grant EDW B13-029-0914 and the MOE grant ERGS13-025-0058.   The first author (M.S.) is
grateful to the Junior Associate scheme of the Abdus Salam
International Centre for Theoretical Physics, Trieste, Italy.

\end{document}